\documentclass[a4paper,12pt,reqno]{amsart}  %%%%% When compiling with pdfLaTeX

\usepackage{amsmath}
\usepackage[all]{xy}
\usepackage{tikz-cd}
\usepackage{amssymb}
\usepackage{amsthm}
\usepackage{amsfonts}
\usepackage{mathtools}
\usepackage{comment}
\usepackage{mathrsfs}
\usepackage{cite}
\usepackage{enumerate}
\usepackage{here}
\usepackage{autobreak}
\usepackage{wrapfig}
\usepackage{graphicx}
\usepackage{lscape}
\usepackage[colorlinks]{hyperref}  %%%%%%%%%% pdfLaTeX ver
\usepackage{xcolor}
\hypersetup{
	bookmarksnumbered=true,
    colorlinks=true,
    citecolor=blue,
    linkcolor=purple,
    urlcolor=orange,
}

\usepackage[capitalize, nameinlink]{cleveref}
\everymath{\displaystyle}
\numberwithin{equation}{section}

\newcommand{\todo}{\textcolor{red}{TODO}} %% TODO
\theoremstyle{plain}

\newtheorem{thm}{Theorem}[section]

\newtheorem{question}[thm]{Question}
\newtheorem{lem}[thm]{Lemma}
\newtheorem{prop}[thm]{Proposition}

\newtheorem{introthm}{Theorem}[section]

\newtheorem{introquestion}[introthm]{Question}

\theoremstyle{definition}

\newtheorem{dfn}[thm]{Definition}

\newtheorem{eg}[thm]{Example}
\newtheorem*{Ack}{Acknowledgement}
\newtheorem*{NoCon}{Notation and Conventions}
\newtheorem*{Out}{Outline of this paper}

\newtheorem*{RelFurth}{Related Work and Further Questions}

\theoremstyle{remark}
\newtheorem*{rem}{Remark}

\DeclareMathOperator{\Ext}{Ext}
\DeclareMathOperator{\ext}{ext}
\DeclareMathOperator{\Pic}{Pic}
\DeclareMathOperator{\Stab}{Stab}
\DeclareMathOperator{\Stabdiv}{Stab_{\mathrm{div}}}
\DeclareMathOperator{\Amp}{Amp}

\DeclareMathOperator{\Spec}{Spec}

\DeclareMathOperator{\ch}{ch}
\DeclareMathOperator{\rk}{rk}
\DeclareMathOperator{\ob}{ob}
\DeclareMathOperator{\id}{id}

\DeclareMathOperator{\pt}{pt}
\DeclareMathOperator{\GL}{GL}
\DeclareMathOperator{\Ku}{Ku}
\DeclareMathOperator{\Cone}{Cone}
\DeclareMathOperator{\Coh}{\mathbf{Coh}}
\DeclareMathOperator{\Qcoh}{\mathbf{Qcoh}}

\DeclareMathOperator{\depth}{depth}
\DeclareMathOperator{\pd}{pd}
\DeclareMathOperator{\gldim}{gl.dim}

\DeclareMathOperator{\hocolim}{hocolim}
\DeclareMathOperator{\thick}{thick}
\DeclareMathOperator{\Perf}{Perf}

\newcommand\Hom{\mathop{\mathrm{Hom}}\nolimits}
\newcommand\RHom{\mathop{\mathbf{R}\mathrm{Hom}}\nolimits}
\newcommand{\PP}[1]{\mathbb{P}^{#1}}

\newcommand\dual{\raise0.9ex\hbox{$\scriptscriptstyle\vee$}}
\newcommand{\hooklongrightarrow}{\lhook\joinrel\longrightarrow}
\newcommand{\Dcpt}{\DC^{\mathrm{cpt}}}

\newcommand{\CB}{\mathbb{C}}
\newcommand{\EB}{\mathbb{E}}

\newcommand{\HB}{\mathbb{H}}
\newcommand{\LB}{\mathbb{L}}
\newcommand{\RB}{\mathbb{R}}

\newcommand{\ZZ}{\mathbb{Z}}

\newcommand{\AC}{\mathcal{A}}
\newcommand{\BC}{\mathcal{B}}
\newcommand{\CC}{\mathcal{C}}
\newcommand{\DC}{\mathcal{D}}
\newcommand{\EC}{\mathcal{E}}
\newcommand{\FC}{\mathcal{F}}
\newcommand{\GC}{\mathcal{G}}
\newcommand{\HC}{\mathcal{H}}

\newcommand{\KC}{\mathcal{K}}

\newcommand{\TC}{\mathcal{T}}

\newcommand{\OO}{\mathscr{O}}

\newcommand{\Rbf}{\mathbf{R}}
\newcommand{\Lbf}{\mathbf{L}}

\author{Tomoki Yoshida}
\address{Department~of~Mathematics, School~of~Science~and~Engineering, Waseda~University, Ohkubo~3-4-1, Shinjuku, Tokyo~169-8555, Japan}
\email{\href{mailto:tomoki_y@asagi.waseda.jp}{tomoki\_y@asagi.waseda.jp}}

\title[Ulrich Objects]{A Simple Derived Categorical Generalization of Ulrich Bundles}
\date{September 14, 2025}

\keywords{Derived category, Ulrich bundles, Algebraic curves}
\subjclass[2020]{14F08 (primary), 14J60, 14D20 (secondary).}
\begin{document}
\begin{abstract}
We define special objects, \emph{Ulrich objects}, on a derived category of polarized smooth projective variety $(X, \OO(1))$ as a generalization of Ulrich bundles to the derived category.
These are defined by the cohomological conditions that are the same form as a cohomological criterion determining Ulrichness for sheaves.

This paper gives a characterization of the Ulrich object similar to the one in \cite{eisenbud_schreyer_2003_resultants_and_chow_forms_via_exterior_syzygies}.
As an application, we have provided a new approach to the Eisenbud-Schreyer question by using the notions of the generator of the derived category. 
We have also given an example of Ulrich objects that are not sheaves by the Yoneda extension.
\end{abstract}
\maketitle
\setcounter{tocdepth}{2} %table of contents subsection
\tableofcontents
\setcounter{section}{-1}
\section{Introduction}\label{Section: introduction}
Let $X$ be a smooth projective variety over $k=\overline{k}$ and $D^b(X)$ be the bounded derived category of coherent sheaves on $X$.
A derived category, whose objects are bounded complexes of coherent sheaves, is a triangulated category constructed from an abelian category.
The investigation of $D^b(X)$ as an invariant of a variety began with the work of Bondal and Orlov \cite{bondal_orlov_2001_reconstruction_of_a_variety_from_the_derived_category_and_groups_of_autoequivalences}, where they showed that if $X$ has (anti-)ample canonical line bundle, then the equivalence $D^b(X)\cong D^b(Y)$ implies $X\cong Y$.

Since then, the derived category has been extensively studied in relation to the minimal model program \cite{kawamata_2002_dequivalence_and_kequivalence} and Bridgeland stability conditions \cite{bridgeland_2009_spaces_of_stability_conditions}, which is a generalized notion of sheaf stability.

Meanwhile, recent research has shown that the derived categories are also valuable to investigate the properties of sheaves.
For instance, derived categorical methods are employed in the works of Ohno \cite{ohno_2020_nef_vector_bundles_on_a_quadric_surface_with_the_first_chern_class_21,ohno_2022_nef_vector_bundles_on_a_projective_space_or_a_hyperquadric_with_the_first_chern_class_small} and Fukuoka, Hara, and Ishikawa \cite{fukuoka_hara_ishikawa_2022_classification_of_rank_two_weak_fano_bundles_on_del_pezzo_threefolds_of_degree_four,fukuoka_hara_ishikawa_2023_rank_two_weak_fano_bundles_on_del_pezzo_threefolds_of_degree_five} to classify certain classes of special bundles. 

This paper proposes that the derived categorical framework is likewise valuable for the study of Ulrich bundles, particularly regarding the existence problem, which has attracted the attention of algebraic geometers in recent years.

\subsection{Ulrich Bundles}\label{subsection in intro: Ulrich bundles}
Let $X\subset\PP{N}$ be a smooth projective variety embedded into the projective space $\PP{N}$. 
Let $\OO_{X}(1) \coloneqq \OO_{\PP{N}}(1)\otimes \OO_{X}$.
A sheaf $\EC$ on $X$ is called an \emph{Ulrich} sheaf if it is initialized, is arithmetically Cohen-Macaulay, and satisfies vanishing conditions of intermediate cohomologies with respect to $\OO_X(1)$ (see \cref{definition: Ulrich bundle}).

In \cite{ulrich_1984_Gorenstein_rings_and_modules_with_high_numbers_of_generators}, Ulrich introduced the notion of Ulrich bundles in the context of commutative algebra. Motivated by Beauville's results \cite{beauville_2000_determinantal_hypersurfaces}, Eisenbud and Schreyer \cite{eisenbud_schreyer_2003_resultants_and_chow_forms_via_exterior_syzygies} later brought this notion to algebraic geometry.
One of the primary issues regarding Ulrich bundles is their existence. 
Eisenbud and Schreyer posed the following question, which is now widely believed to have a positive answer and is known as Eisenbud–Schreyer conjecture.
\begin{introquestion}[\cite{eisenbud_schreyer_2003_resultants_and_chow_forms_via_exterior_syzygies}]\label{Question: the Eisenbud-Schreyer's conjecture}
    Is every smooth projective variety $X\subset\PP{N}$ the support of an Ulrich sheaf? 
\end{introquestion}

The existence problem has been studied for some varieties through ad hoc methods.
For example, the Ulrich bundles on the projective spaces are given by (direct sums of) structure sheaves, and these on the quadric hypersurface $Q^n\subset\PP{n+1}$ by Spinor bundle(s). 
For more general hypersurfaces, their existences are studied via matrix factorizations.
Additionally, the proof of existence often relies on the choice of a very ample divisor on $X$.

\subsection{Ulrich Objects}\label{subsection in intro: Ulrich objects}
In this paper, we introduce the notion of \emph{Ulrich objects} as a generalization of Ulrich bundles in terms of the derived category.
These objects are defined by cohomological conditions analogous to the cohomological criterion of Ulrichness for sheaves (see \cref{definition: Ulrich object}).
Firstly, we begin with establishing a characterization of Ulrich objects that is partially parallel to the characterization of Ulrich bundles (\cref{theorem: characterization of Ulrich bundle}).

Let $X$ be a smooth projective variety with an embedding $X\hookrightarrow \PP{N}$ by a linear system $|\OO_X(1)|$ and $\pi: X \to \PP{\dim X}$ a finite linear projection (\cref{definition: finite linear projection}). 

\begin{introthm}[see \cref{Proposition: Ulrich object characterization}]\label{Theorem in Introduction: Ulrich object characterization}
With the setting above, the following conditions are equivalent:
    \begin{enumerate}
        \item $E\in D^b(X)$ is an Ulrich object, 
        \item $\pi_{\ast}E\in \langle\OO\rangle$, 
        \item each cohomology sheaf $\HC^i(E)$ is an Ulrich bundle if it is not zero.
    \end{enumerate}
\end{introthm}

The last condition in the Theorem seems to mean that one does not get something new by this generalization.
However, the last condition in \cref{Theorem in Introduction: Ulrich object characterization} plays an important role in the study of the existence of Ulrich objects and bundles.
Namely, the existence of the Ulrich object is equivalent to the existence of the Ulrich bundle. 

Also, we present an example of an Ulrich object that is not an Ulrich bundle: \emph{a Yoneda type Ulrich object}.

\subsection{Application to the Eisenbud-Schreyer's Question}\label{subsection in intro: application}
As stated above, to establish the existence of Ulrich bundles on $(X, \OO(1))$, it suffices to show the existence of Ulrich objects.
In terms of the derived category, the existence of the Ulrich object is equivalent to the following claim:
\[
\GC' \coloneqq \OO_X(1)\oplus \cdots \oplus \OO_X(n) \text{ is not a generator}.
\]
In \cref{subsection: generators of the derived category}, we will review the definitions of various types of \emph{generators in triangulated categories}, as well as the relationship among them.

We present a new general proof of the existence of Ulrich bundles on an elliptic curve 
with an arbitrary embedding, based on the framework of the derived category.
By \cref{Theorem in Introduction: Ulrich object characterization}, known results, and a few discussions deduce the following corollary.

\begin{introthm}[see \cref{theorem: curve case existence of Ulrich object}]\label{Theorem in introduction: curve Ulrich object}
    For any elliptic curve with an embedding $(E, \OO_{E}(1))$, 
    there exists an Ulrich object. 
\end{introthm}
Note that the existence of the Ulrich line bundle on a curve has already been established in \cite{eisenbud_schreyer_2009_betti_numbers_of_graded_modules_and_cohomology_of_vector_bundles}, 
see also \cite{coskun_2017_a_survey_of_ulrich_bundles} for a survey.

\begin{RelFurth}
The existence problem has been widely studied, and there are too many references to list them out here, but we mention some of them:

As stated above, the curve case treated in \cite{eisenbud_schreyer_2009_betti_numbers_of_graded_modules_and_cohomology_of_vector_bundles}.
The results for hypersurfaces and, more generally, complete intersections are studied in \cite{herzog_ulrich_backelin_1991_linear_maximal_cohenmacaulay_modules_over_strict_complete_intersections,brennan_herzog_ulrich_1987_maximally_generated_cohenmacaulay_modules}.
Also, Grassmannians are discussed in \cite{costa_miroroig_2015_glvinvariant_ulrich_bundles_on_grassmannians}.
\cite{feyzbakhsh_pertusi_2023_serreinvariant_stability_conditions_and_ulrich_bundles_on_cubic_threefolds} discusses the property of Ulrich bundles via the Bridgeland stability conditions.
\end{RelFurth}

\begin{Out}
\cref{section: towards the definition of Ulrich Objects} reviews the preliminary results on Ulrich bundles and some basic notions of the bounded derived category of coherent sheaves for clarifying the definition of the Ulrich object.
The Ulrich object will be introduced in \cref{section: Ulrich objects definition and first properties}. Also, \cref{section: Ulrich objects definition and first properties} contains the proof of \cref{Theorem in Introduction: Ulrich object characterization}.

The following \cref{section: examples of Ulrich object} discusses the example of Ulrich objects. In particular, we introduce a special type of Ulrich object, called a Yoneda-type Ulrich object.
The application of Ulrich object to \cref{Question: the Eisenbud-Schreyer's conjecture} is given in \cref{section: applications}. 
Especially, the proof of \cref{Theorem in introduction: curve Ulrich object} is in \cref{subsection: existence for curves}. 
%Finally, some remarks about the relationship between the Ulrich objects and (divisorial) Bridgeland stability conditions are in \cref{section: Bridgeland stability of Ulrich objects}.
\end{Out}

\begin{NoCon}
The conventions and notations used in this paper are listed below:
\begin{itemize}
    \item all varieties are smooth and projective defined over $k=\overline{k}$, 
    \item the dimension of the variety $X$ is written by $n$ unless otherwise stated,
    \item $\DC$ represents a triangulated category with arbitrary direct sums, 
    \item $\Dcpt$ is a full subcategory of $\DC$ that consists of compact objects, 
    \item for an abelian category $\AC$, $D^b(\AC)$ is the bounded derived category of $\AC$,
    \item the bounded derived category $D^b(\Coh(X))$ is denoted by $D^b(X)$ for short,
    \item $D(\Qcoh(X))$ is the unbounded derived category of quasi-coherent sheaves on the scheme $X$,
    \item $\Perf(X)$ is a triangulated category consisting of perfect complexes on $X$ (when $X$ is smooth, $D(\Qcoh(X))^{\textrm{cpt}} = \Perf(X) = D^b(X)$), and
    \item all functors are derived unless otherwise stated. However, $\Lbf$ and $\Rbf$ are occasionally used to emphasize that the functor is actually derived.
\end{itemize}

\end{NoCon}

\begin{Ack}
    The author is grateful to his advisor, Professor Yasunari Nagai, for his continuous encouragement and fruitful discussions.
    He would also like to thank Professor Yuki Mizuno for useful discussions and Professors Wahei Hara and Dmitri Pirozhkov for pointing out the mistake in Proposition 4.9 in the previous version of this paper.
    The author was supported by JST SPRING, Grant Number JPMJSP2128.
\end{Ack}
\section{Towards the Definition of Ulrich Object}\label{section: towards the definition of Ulrich Objects}
This \cref{section: towards the definition of Ulrich Objects} contains the preliminary results of Ulrich bundles and derived categories of coherent sheaves for non-experts of each topic.
Thus, there is no problem for experts to skip \cref{section: towards the definition of Ulrich Objects}.

\subsection{Preliminaries on Ulrich Bundles}\label{subsection: Preliminaries on Ulrich Bundles}
Most of the basic results of Ulrich bundles are cited from the book about Ulrich bundles \cite{book_costa_miroroig_2021_ulrich_bundlesfrom_commutative_algebra_to_algebraic_geometry}.
Also, \cite{beauville_2018_an_introduction_to_ulrich_bundles} is an article for an introduction to Ulrich bundles.
\begin{dfn}\label{definition: aCM bundle}
    A sheaf $\EC$ on $X$ is called \emph{arithmetically Cohen-Macaulay} (aCM for short) if
    locally Cohen-Macaulay (i.e. $\depth\EC_x = \dim\OO_{X, x}$ for any $x\in X$) and
    $H^i(X, \EC(t)) = 0$ for any $t\in\ZZ$ and $i= 1, \dots, n-1$.
\end{dfn}
\begin{dfn}\label{definition: initialized bundle}
    A sheaf $\EC$ on $X$ is called \emph{initialized} if 
    $H^0(X, \EC) \neq 0$ and $H^0(X, \EC(t)) = 0$ for any $t<0$.
\end{dfn}
\begin{dfn}\label{definition: Ulrich bundle}
    A sheaf $\EC$ on $X$ is called \emph{Ulrich} if
    it is initialized aCM and $h^0(X, \EC) = \deg X\rk\EC$. 
\end{dfn}
\begin{rem}
    If $\EC$ is initialized and aCM, the inequality holds;
    \[
    h^0(\EC) \le \deg(X)\rk\EC.
    \]
\end{rem}
\begin{rem}
    Any aCM bundle is a vector bundle.
    Indeed, by the Auslander-Buchsbaum formula, 
    \[
    \pd\EC_x = \depth\OO_{X, x} -\depth\EC_x =0.
    \]
    Thus, $\EC_x$ is $\OO_{X, x}$- free module for any $x\in X$, 
    noting that the smoothness of the variety is used here.
    Thus, Ulrich sheaves are sometimes referred to as Ulrich bundles. 
\end{rem}
\begin{dfn}\label{definition: finite linear projection}
    For an polarized variety $(X, \OO(1))$, 
    the morphism $\pi: X\to \PP{n}$ as following diagram:
    \[
\begin{tikzcd}
X \arrow[rr, "\Phi_{|\OO(1)|}", hook] \arrow[rrd, "\pi"] &  & \PP{N} \arrow[d, "p", dashed] \\
& &\PP{n},      
\end{tikzcd}
\]
where the $p$ is the composition of linear projection from the points in $\PP{N}\setminus X$.
The composition $\pi\coloneqq p\circ\Phi_{|\OO(1)|}$ that is defined above is called \emph{finite linear projection}.
\end{dfn}
\begin{rem}
    The finite linear projection is a finite morphism.
    Indeed, as the linear projections are quasi-finite and proper, this is a finite morphism.
\end{rem}
The motivation for considering the Ulrich bundles goes back to the paper \cite{beauville_2000_determinantal_hypersurfaces}.
Let $X\subset\PP{n+1}$ be a hypersurface of degree $d$, defined by the equation $F=0$.
Then, Beauville showed that 
$X$ behaves like a determinantal variety in the meaning of  $F^r = \det M$ for some $r$ and $M$
if and only if 
there exists a rank $r$ vector bundle $\EC$ on $X$ that admits 
a linear resolution 
\begin{equation}
    0 \to \OO_{\PP{n+1}}(-1)^{\oplus rd} \stackrel{M}{\longrightarrow} \OO_{\PP{n+1}}^{\oplus rd} \to \EC \to 0.
    \notag
\end{equation}
Generalizing this result, Eisenbud and Schreyer have shown that the following result: 
\begin{thm}[{\cite[Proposition 2.1.]{eisenbud_schreyer_2003_resultants_and_chow_forms_via_exterior_syzygies}}]\label{theorem: characterization of Ulrich bundle}
    Let $\EC$ be an initialized vector bundle on $n$-dimensional variety $X$ and an embedding $\Phi_{|\OO_X(1)|} : X\hookrightarrow\PP{N}$.
    The following statements are equivalent:
    \begin{enumerate}
        \item $\EC$ is an Ulrich bundle,
        \item There is a linear resolution
        \begin{equation}
            0 \to \OO_{\PP{N}}(-N+n)^{\oplus a_{N-n}} \to \cdots \to \OO_{\PP{N}}(-1)^{\oplus a_{1}} \to \OO_{\PP{N}}^{\oplus a_{0}} \to \EC \to 0
            \notag
        \end{equation}
        \item $H^i(\EC(-j)) = 0$ for any $i\ge 0$ and $j= 1, \dots, n$, 
        \item For some finite linear projection $\pi : X\to \PP{n}$, 
        $\pi_{\ast}\EC\cong\OO^{\oplus m}_{\PP{n}}$ for some $m$.
    \end{enumerate}
\end{thm}

The rest of this subsection is devoted to reviewing the stability of Ulrich bundles.

\begin{thm}[{\cite[Theorem 2.9.]{casanellas_hartshorne_geiss_2012_stable_ulrich_bundles}}]\label{Theorem: Ulrih bundles are Gieseker stable}
    Let $(X, \OO(1))$ be a $n$-dimensional variety and 
    $E$ be an Ulrich bundle on $X$.
    Then, $E$ is semistable in the sense of Gieseker 
    (In particular, $\mu$-semistable) with respect to $\OO(1)$.
\end{thm}

\begin{prop}
    Let 
    \[
    0=\EC_0\subset \cdots \subset \EC_{n-1}\subset\EC_{n}=\EC
    \]
    be a Jordan-H\"older filtration of Ulrich bundle $\EC$.
    Then, each sub-sheaf $\EC_{i}$ is also an Ulrich bundle. 
\end{prop}

\subsection{Concise reviews on Derived Categories}\label{subsection: concise reviews on derived categories}
In \cref{subsection: concise reviews on derived categories}, we shortly give reviews on some definitions and notations of derived categories of coherent sheaves for non-experts of derived categories.
The derived category of an abelian category has the structure of a triangulated category.
Some of the following results can be applied to triangulated categories. 
However, we will only mention the result for derived categories for simplicity.

Let $X$ be a smooth projective variety of dimension $n$ over $\CB$, and 
$D^b(X)\coloneqq D^b(\Coh(X))$ be the bounded derived category of coherent sheaves.
The object $E\in D^b(X)$ is, as an entity, the complex of quasi-coherent sheaves on $X$ with coherent cohomologies like:

\begin{equation}
    E: \big(\cdots \to E_{-1} \to E_0 \to E_1 \to E_2 \to \cdots\big). \notag
\end{equation}
The $i$-th cohomology sheaf of $E$ as a complex is written as $\HC^i(E)\in\Coh(X)$ and 
$H^i(E) \coloneqq \HC^i(\Rbf\Gamma(E))\in\mathbf{Vec}$, here $\Rbf\Gamma: D^b(X) \to D^b(\mathbf{Vec})$ is the derived functor of global section functor.

\begin{dfn}\label{definition: thick closure and thick subcategory}
    Let $\CC\subset\DC$ be a full subcategory or a set of objects. 
    \begin{enumerate}
        \item The \emph{thick closure} of $\CC$ in $\DC$ is the smallest full triangulated subcategory containing $\CC$ which is closed under taking direct summands.
        Denote this by $\thick\CC$.
        \item The full subcategory $\CC$ is called \emph{thick} if $\CC = \thick\CC$ holds.
        Denote the smallest thick subcategory containing $\CC$ by $\langle \CC \rangle$.
    \end{enumerate}
\end{dfn}

\begin{dfn}\label{definition: Semi-orthogonal decommposition}
		Let $\DC$ be a triangulated category. 
		A sequence of full subcategories $\{\mathcal{C}_{1},\ldots, \mathcal{C}_{l}\}$ forms
		a \emph{semi-orthogonal decomposition} of $\DC$ if the following two conditions hold: 
	\begin{enumerate}[(1)]
		\item $\Hom(\CC_{m}, \CC_{n}) = 0$ for $m>n$, 
		\item for any object $D\in\DC$, there is a sequence
		\[
            0=D_{l}\rightarrow D_{l-1}\rightarrow\cdots\rightarrow D_{1}\rightarrow D_{0}=  D, 
        \]
			such that $\Cone(D_{i}\rightarrow D_{i-1})\in\CC_{i}$.
	\end{enumerate}
	Denote the semi-orthogonal decomposition by 
    $\DC = \langle\CC_{1},\ldots,\CC_{l}\rangle$ (Indeed, in this case, the thick closure of $\CC_{1},\ldots,\CC_{l}$ coincides with $\DC$). 
\end{dfn}

\begin{dfn}
    A full triangulated subcategory $\CC\subset D^b(X)$ is called \emph{admissible} if
    the natural inclusion functor $\iota_{\ast}: D^b(X) \to D^b(X)$ admits 
    both of adjoints $\iota^{\ast} \dashv \iota \dashv \iota^{R}$.
\end{dfn}
Note that for a triangulated category with the Serre functor, any subcategory that can be a component of semi-orthogonal decomposition of $\DC$ is admissible.
Conversely, for a admissible subcategory $\CC\subset\DC$, there are the semi-orthogonal decompositions
\[
\DC = \langle\AC, \CC\rangle =\langle\CC, \BC\rangle.
\]
Next, we introduce the \emph{(full) exceptional collection}, which is a most simple form of semi-orthogonal decomposition.
\begin{dfn}\label{definition: exceptional objects}
    Let $\DC$ be a triangulated category and $E$ an object in $\DC$. 
    $E$ is called \emph{exceptional} if the following condition holds:
    \begin{equation*}
        \hom^{i}(E, E) = 
        \begin{cases}
            1 & (i=0), \\
            0 & (i\neq0).
        \end{cases}
    \end{equation*}
\end{dfn}

\begin{dfn}\label{Definition: exceptional collection}
    The ordered collection of exceptional objects $(E_{1},\ldots, E_{n}) \subset \DC$
    is called an \emph{exceptional collection} if 
    for any $i,j\in\{1,\dots ,n\}$ with $i<j$ and for any $k\in\ZZ$, 
    $$\Hom^{k}(E_{i}, E_j) = 0. $$

    An exceptional collection $\mathbb{E} = (E_{1},\ldots, E_{n})$ is called 
    \begin{enumerate}
        \item \emph{full} if
        $\EB$ generates the derived category by shifts and extensions;
        \item \emph{$\Ext$-exceptional} if
        $\Hom^{k}(E_i, E_j) = 0$ for all $i\neq j$, when $k\le0$;
        \item \emph{strong} if 
        $\Hom^{k}(E_i, E_j) = 0$ for all $i$ and $j$, when $k\neq 0$.
    \end{enumerate}
\end{dfn}

\begin{eg}\label{Example: Beilinson collection and left dual}
    The projective space $\PP{n}$ has a strong full exceptional collection 
    \[
        D^{b}(X) = \EB \coloneqq (\OO, \OO(1), \dots, \OO(n)), 
    \]
    that is called \emph{Beilinson Collection}.
\end{eg}

\begin{comment}
In general, the semi-orthogonal decomposition of a derived category is, if it exists, not unique.
The following functor,\emph{mutation functor}, is a procedure that makes a new semi-orthogonal decomposition from another one.
\begin{dfn}\label{definition: Mutation Functor}
    Let $\CC$ be an admissible subcategory of the derived category $D^b(X)$.
    The functor that is defined by 
    \begin{align}
        %\RB_{\CC}&: D^b(X) \to D^b(X) \notag\\
        \LB_{\CC}: D^b(X) &\to D^b(X)\notag\\
         F &\mapsto \Cone(\iota\circ\iota^{R}(F) \to F)
        \notag
    \end{align}
    is called (left) mutation functor. 
\end{dfn}
	\begin{prop}\label{proposition: property of mutation functor}
		Suppose that there is a semi-orthogonal decomposition 
		$\mathcal{D} = \langle\mathcal{C}_{1}, \ldots, \mathcal{C}_{l}\rangle$. 
		Then, a mutation gives the following semi-orthogonal decompositions 
		\[
            \DC = \langle\CC_{1}, \ldots, \CC_{i-1}, 
			\LB_{\CC_{i}}(\CC_{i+1}), \CC_{i}, \ldots, \CC_{l}\rangle. 
       \]
	\end{prop}
\end{comment}

\section{Ulrich Objects: Definition and First Properties}\label{section: Ulrich objects definition and first properties}

In this \cref{section: Ulrich objects definition and first properties}, 
we generalize the notion of Ulrich bundles to the derived categorical setting by the characterization of Ulrich bundles \cref{theorem: characterization of Ulrich bundle}.
We refer to the basic results about the derived categories to \cite{book_huybrechts_2006_fouriermukai_transforms_in_algebraic_geometry}.
It will start with the definition of Ulrich objects.
\begin{dfn}\label{definition: Ulrich object}
    Let $X$ be an algebraic variety of dimension $n$, $\OO_{X}(1)$ be a very ample divisor on $X$, 
    and $E$ be an object of the bounded derived category $D^b(X)$.
    
    $E$ is said to be an \emph{Ulrich object} with respect to the polarization $\OO_{X}(1)$ if 
    \begin{equation}
       H^i(E(-j))=0 \notag 
    \end{equation}
    for any $i\in \ZZ$ and $1\le j\le n$.
\end{dfn}
    Note that the Ulrich bundle embedded into $D^b(X)$ as a complex concentrated in a term is an Ulrich object.
\begin{comment}
\begin{dfn}\label{definition: Ulrich bundle indecomposable}
    An Ulrich bundle $E$ on $X$ is called \emph{decomposable} 
    if there exist non-zero \emph{Ulrich objects} $E_1$ and $E_2$ in $D^b(X)$ such that $E\cong E_1\oplus E_2$.
    An Ulrich bundle on $X$ is called \textit{indecomposable} if it is not decomposable.
\end{dfn}
\end{comment}

\begin{prop}\label{Proposition: Ulrich object characterization}
    Let $E\in D^b(X)$ be an Ulrich object on a smooth projective variety with an embedding $(X, \OO_X(1))$ and 
    $\pi: X\to \PP{n}$ a finite linear projection (see \cref{definition: finite linear projection}).
    Then, the following conditions are equivalent:
    \begin{enumerate}
        \item $E\in D^b(X)$ is an Ulrich object, 
        \item $\pi_{\ast}E\in \langle\OO_{\PP{n}}\rangle$, 
        \item each cohomological sheaf $\HC^i(E)$ is an Ulrich bundle (if not zero).
    \end{enumerate}
\end{prop}
\begin{proof}[\proofname\ of \cref{Proposition: Ulrich object characterization}]
    Let us prove that (1) implies (2).
    By the construction of $\pi$, 
    $\pi^{\ast}\OO_{\PP{n}}(1)\cong\OO_X(1)$.
    Thus, for any $i\in\ZZ$ and for $j= 1, \dots, n$,
    \begin{align*}
        \Hom_{D^b(\PP{n})}^i(\OO_{\PP{n}}(j), \pi_{\ast}E) 
        &\cong\Hom_{D^b(X)}^i(\pi^{\ast}\OO_{\PP{n}}(j), E)\\
        &\cong\Hom_{D^b(X)}^i(\OO_X(j), E) \\
        &\cong H^i(E(-j)) = 0.
    \end{align*}
    Therefore, 
    \[
    \pi_{\ast}E \in \langle \OO(1), \dots, \OO(n)\rangle^{\perp} 
    = \langle \OO \rangle, 
    \]
    that proves what we want.

    Next, we will prove that (2) implies (3).
    As higher direct images of $\Rbf\pi_{\ast}$ vanish since $\pi$  is finite, 
    $\Rbf\pi_{\ast}E$ is the form of
    \begin{equation}
        \dots \to \pi_{\ast}E_{-1} \to \pi_{\ast}E_{0} \to \pi_{\ast}E_{1} \to \cdots \ .
        \notag
    \end{equation}
    This complex splits by the assumption (2) and $\langle \OO_{\PP{n}} \rangle \cong D^b(\Spec\CB)$, 
    that is, 
    \[
    \pi_{\ast}E \cong \bigoplus_{i\in\ZZ} \OO_{\PP{n}}^{m_i}[-i].
    \]
    This shows $\pi_{\ast}\HC^{i}(E)\cong \OO_{\PP{n}}^{m_i}$.
    Note that here the fact is used that $\Rbf\pi_{\ast}$ commute with $\HC^i(-)$ (cf. \cite[Corollary 3.3.4.]{bondal_van_2003_generators_and_representability_of_functors_in_commutative_and_noncommutative_geometry}). 
    Thus, the characterization of Ulrich bundle \cref{theorem: characterization of Ulrich bundle} claims that each $\HC^{i}(E)$ is Ulrich bundle unless $m_i = 0$.

    Considering the spectral sequence 
    \[
    E_2^{p,q} = H^{p}(\HC^q(E)(-j)) \Longrightarrow H^{p+q}(E(-j)), 
    \]
    we have that (3) implies (1).
\end{proof}
A generalization of the equivalence of (1) and (2) in \cref{Proposition: Ulrich object characterization} is the following proposition, which is a well-known statement for the case of Ulrich bundles.

\begin{prop}\label{proposition: finite morphism Ulrich obj}
    Let $\pi: X \to Y$ be a finite morphism between the same dimensional smooth projective varieties, 
    and $\OO_Y(1)$ be a very ample line bundle on $Y$. 
    Assume that $\pi^{\ast}\OO_Y(1)$ is also very ample on $X$.
    Then, $E$ is Ulrich object with respect to $\pi^{\ast}\OO_Y(1)$
    if and only if
    $\pi_{\ast}E$ is Ulrich bundle with respect to $\OO_Y(1)$.
\end{prop}
\begin{proof}[\proofname\ of \cref{proposition: finite morphism Ulrich obj}]
    Considering the Leray's spectral sequence
    \begin{equation}
        E_{2}^{p,q} = H^q(Y, R^{p}\pi_{\ast}E(-j)) \Rightarrow H^{p+q}(X, E(-j))
        \notag
    \end{equation}
    and higher direct image vanishes because $\pi$ is finite, 
    we have the claim. 
\end{proof}

\begin{prop}\label{proposition: hyperplane section of Ulrich object}
    Let $Y\in |H|$ be a general hyperplane section of $X$, $E$ an Ulrich object on $X$, 
    and $\OO_Y(1)\coloneqq \OO_X(1)|_{Y}$. 
    Then, $E|_Y$ is also an Ulrich object for $(Y, \OO_Y(1))$.
\end{prop}
\begin{proof}[\proofname\ of \cref{proposition: hyperplane section of Ulrich object}]
    Considering the triangle
    \begin{equation*}
        E(-j-1) \to E(-j) \to E|_Y(-j) \to E(-j-1)[1]
    \end{equation*}
    and setting $j=1, \dots, n-1$, 
    $H^i(E|_Y(-j)) = 0$ hold for any $i\in\ZZ$ and $j = 1,\dots, n-1$.
\end{proof}

\begin{lem}\label{lemma: Ulrich object 2 out of 3}
    Let $E, F, G$ be objects of $D^b(X)$ such that there exists the distinguished triangle
    \[
    E\to F\to G\to E[1].
    \]
    If two of $E, F, G$ are Ulrich objects, then the other is. 
\end{lem}
\begin{proof}[\proofname\ of \cref{lemma: Ulrich object 2 out of 3}]
    For each $j=1, \dots, n$, 
    \[
    E(-j)\to F(-j)\to G(-j)\to E(-j)[1]
    \]
    is a distinguished triangle. Thus, the claim can be concluded by taking the long exact sequence. 
\end{proof}

\begin{prop}\label{proposition: Ulrich object on product}
    Let $E$ (resp. $F$) be an Ulrich object on $(X, \OO_X(1))$ (resp. $(Y, \OO_Y(1))$). 
    Then, $E \boxtimes F(\dim X)$ and $E(\dim Y)\boxtimes F$ 
    are Ulrich objects on $(X\times Y, \OO_X(1)\boxtimes \OO_Y(1))$.
\end{prop}
\begin{proof}[\proofname\ of \cref{proposition: Ulrich object on product}]
    Let $n=\dim X$. Let us show the case of $E \boxtimes F(\dim X)$.
    The claim easily follows from the K\"{u}nneth formula 
    \[
    H^i(X\times Y, E\boxtimes F(n)\otimes \OO_X(-j)\boxtimes \OO_Y(-j))
    = \bigoplus_{p+q=i} H^p(X, E(-j))\otimes H^q(Y, F(n-t)).
    \]
\end{proof}

\section{Ulrich Objects: Examples}\label{section: examples of Ulrich object}
\subsection{Yoneda Type Ulrich Objects}\label{subsection: Yoneda type Ulrich objects}

The next subject of consideration shifts to the Yoneda extension.
We refer \cite{oort_1964_yoneda_extensions_in_abelian_categories} for the basic notions and results of the Yoneda extension.

\begin{dfn}\label{definition: Yoneda extension}
    Let $\FC$ and $\GC$ be sheaves on $X$.
    Assume that there exists a non-trivial element $\eta\in\Ext^{k}(\FC, \GC)$ for some $k\ge1$.
    Then, there is the following exact sequence corresponding to $\eta$:
    \begin{equation}
        0 \to \GC \to E_{k-1} \to \cdots \to E_0 \to \FC \to 0.
        \notag
    \end{equation}

In this paper, we refer to the object 
\begin{equation}
    E_{\eta}: (\cdots 0 \to E_{k-1} \to \cdots \to E_0 \to 0 \cdots) \in D^b(X) \notag
\end{equation}
as the \emph{Yoneda extension} of degree $k$ corresponds to $\eta\in\Ext^{k}(\FC, \GC)$
(here, the degree of $E_j$ as the complex is the same as $j$).

    Note that two Yoneda extensions 
    \begin{equation}
        0 \to \GC \to E_{k-1} \to \cdots \to E_0 \to \FC \to 0,\notag
    \end{equation}
    and
    \begin{equation}
        0 \to \GC \to F_{k-1} \to \cdots \to F_0 \to \FC \to 0 \notag
    \end{equation}
    are said to be equivalent if there exists another Yoneda extension 
    \begin{equation*}
        0 \to \GC \to G_{k-1} \to \cdots \to G_0 \to \FC \to 0
    \end{equation*}
    such that there is the following commutative diagram:
    \begin{equation*}
\begin{tikzcd}
0 \arrow[r] & \GC \arrow[r]& E_{k-1} \arrow[r] & \cdots \arrow[r] & E_1 \arrow[r] & E_0 \arrow[r]         & \FC \arrow[r]  & 0 
\\
0 \arrow[r] & \GC \arrow[r] \arrow[u, "\id"'] \arrow[d, "\id"] & G_{k-1} \arrow[r] \arrow[d] \arrow[u] & \cdots \arrow[r] & G_1 \arrow[r] \arrow[u] \arrow[d] & G_0 \arrow[r] \arrow[d] \arrow[u] & \FC \arrow[r] \arrow[u, "\id"] \arrow[d, "\id"'] & 0 
\\
0 \arrow[r] & \GC \arrow[r] & F_{k-1} \arrow[r]& \cdots \arrow[r] & F_1 \arrow[r] & F_0 \arrow[r]& \FC \arrow[r]  & 0.
\end{tikzcd}
\end{equation*}
\end{dfn}

The Yoneda extension $E$ is isomorphic to $\GC[-k+1]\oplus\FC$ if and only if 
$E$ is corresponding to $0\in\Ext^{k}(\FC, \GC)$.
In particular, it is observed that the Yoneda extension gives a non-trivial example of Ulrich object: 
\begin{eg}\label{example: Yoneda type Ulrich object}
    Let $\FC$ and $\GC$ be Ulrich bundles. 
    Assume that there is a non-trivial element $\eta\in\Ext^{m}(\FC, \GC)$
    for $m\ge1$.
    Then, Yoneda extension defined in \cref{definition: Yoneda extension}
    \begin{equation}
        E_{\eta}: \cdots \to 0\to E_{-m+1}\to \cdots \to E_{0} \to 0 \to \cdots \notag 
    \end{equation}
    correspond to $\eta$ is an Ulrich object. 

    Indeed, because
    \begin{equation}
        \HC^i(E) = 
        \begin{cases}
            \FC & \text{if } i= 0,\\
            \GC & \text{if } i= -m+1,\\
            0 & \text{otherwise},
        \end{cases}
        \notag
    \end{equation}
    it is shown from \cref{Proposition: Ulrich object characterization}.
\end{eg}

\begin{dfn}\label{definition: Yoneda type Ulrich objects}
    We call the Ulrich object constructed in the way of \cref{example: Yoneda type Ulrich object} (up to shifts) as \emph{Yoneda type Ulrich object}.
\end{dfn}

\begin{comment}
\begin{eg}
    Let $\EC, \FC, \GC$ be Ulrich objects on $(X, \OO_{X}(1))$. 
    Assume there exist non-zero elements $\alpha\in\Ext^2(\EC, \FC)$ and $\beta\in\Ext^3(\EC, \GC)$.

    Write the Yoneda type Ulrich object associated to $\alpha, \beta$ by
    \begin{align}
        0 \to F_{-1} \stackrel{f_{0}}{\to} F_{0} \to 0,  \notag \\
        0 \to G_{-2} \stackrel{g_{-1}}{\to} G_{-1} \stackrel{g_0}{\to} G_{0} \to 0. \notag
    \end{align}
    Then, the following exact sequence 
    \begin{equation}
        0 \GC \to G_{-2} \stackrel{0\oplus g_{-1}}{\longrightarrow} F_{-1}\oplus G_{-1} \stackrel{f_0 \oplus g_0}{\longrightarrow} F_{0}\oplus G_{0} \stackrel{f+g}{\longrightarrow} \EC \to 0
        \notag
    \end{equation}
    is an Ulrich as 
    \begin{equation}
        \HC^i(E) = 
        \begin{cases}
            \EC & i=0, \\
            \FC & i=, \\
            \GC & i=, 
        \end{cases}
    \end{equation}
\end{eg}

cf: \cite[Proposition 3.2.]{ishii_uehara_2005_autoequivalences_of_derived_categories_on_the_minimal_resolutions_of_ansingularities_on_surfaces}
\end{comment}

\subsection{Explicit Examples}
First, directly from \cref{Proposition: Ulrich object characterization}, 
the object $E\in D^b(\PP{n})$ is Ulrich with respect to $\OO(1)$ if and only if 
$E\in\langle\OO\rangle$.

When the dimension of the variety is $1$, Ulrich objects are just a direct sum of Ulrich bundles from \cref{Proposition: Ulrich object characterization}.
    Indeed, in this case, any object $E$ is formal, that is, $E$ can be written as 
    \[
    E\cong \bigoplus_{i\in\ZZ}\HC^i(\EC)[-i]
    \]
    (see \cite[Corollary 3.15]{book_huybrechts_2006_fouriermukai_transforms_in_algebraic_geometry}). 
    Thus, any indecomposable Ulrich object on $(C, \OO(1))$ is only an Ulrich bundle.

\subsubsection{Ulrich Objects on Quadric Hypersurfaces}\label{subsubsection: examples the case of quadric Hypersurfaces}
The Ulrich bundles on quadric hypersurfaces $Q_n$ are bundles known as \emph{Spinor bundles}.
The original reference of spinor bundles on quadrics is \cite{ottaviani_1988_spinor_bundles_on_quadrics}.
The derived categorical treatment of quadrics are in \cite[Section 4.]{kapranov_1988_on_the_derived_categories_of_coherent_sheaves_on_some_homogeneous_spaces}.
\begin{thm}[\cite{kapranov_1988_on_the_derived_categories_of_coherent_sheaves_on_some_homogeneous_spaces}]\label{Theorem: SOD of quadric hypersurface Kapranov collection}
    Let $Q_n\subset \PP{n+1}$ be a quadric hypersurface.
    Then, there is a full exceptional collection
    \begin{equation}
        D^{b}(Q_n) = 
        \begin{cases}
            \langle \OO, S, \OO(1), \dots, \OO(n-1) \rangle &\text{if } n:\text{odd,}\\
            \langle \OO, S^+, S^-, \OO(1), \dots, \OO(n-1) \rangle &\text{if } n:\text{even,}
        \end{cases}
        \notag
    \end{equation}
    where $S$, $S^+$, and $S^-$ are the spinor bundles.
\end{thm}
The Kapranov collection above claims that 
$E\in D^b(Q_n)$ is Ulrich object if and only if 
$E\in\langle S \rangle$ when $n$ is odd and 
$E\in\langle S^+, S^- \rangle$ when $n$ is even.

More precisely, 
when $n$ is odd Ulrich object $E$ is of the form of 
\begin{equation}
E\cong \bigoplus_{i\in\ZZ} S^{m_i}[-i]. \notag
\end{equation}
Noting that $S^+$ and $S^-$ are orthogonal, 
when $n$ is even Ulrich object $E$ is of the form of 
\begin{equation}
    E\cong \bigoplus_{i\in\ZZ} (S^{-})^{m_i}[-i] \oplus \bigoplus_{j\in\ZZ} (S^{+})^{m_j}[-j].\notag
\end{equation}

\subsubsection{Ulrich Objects on Surfaces}\label{subsubsection: examples the case of Surfaces}

In the case of curve $C$, any object in $D^b(C)$ splits into the direct sum of shifts of sheaves, but in the case of surface, the following proposition is known by Uehara and Ishii:

\begin{prop}[{\cite[Proposition 3.2.]{ishii_uehara_2005_autoequivalences_of_derived_categories_on_the_minimal_resolutions_of_ansingularities_on_surfaces}}]\label{proposition: Uehara Ishii objects in der cats of surfaces}
    Let $S$ be a surface.
    Giving a object in $D^b(C)$ is equivalent to giving finitely many sheaves $\GC^i$ and element 
    \[
    e^i\in\Ext^2(\GC^i, \GC^{i-1})
    \]
    for each $i$.
\end{prop}

Thus, giving an Ulrich object on a surface $S$ is equivalent to giving finitely many
Ulrich sheaves $\GC^i$ and their degree $2$ extensions.

    Moreover, the Ulrichness of the object imposes a strong restriction on its Chern character.
    To see this, for example, let $X$ be a surface of Picard rank $1$.
    Let us write $\Pic(X) = \ZZ[H]$, $H^2 = d$, and $K_X = i_XH$.

\begin{lem}\label{lemma: Chern polynomial of Ulrich objects}
    Let $E$ be an Ulrich object of rank $r$ on $X$. Then, 
    \begin{align}
        \ch([E]) = r +\frac{r}{2}(i_X+3)H + \bigg( -r\chi(\OO_X) + \frac{rd}{4}(i_X^2+3i_X+4) \bigg).\label{tag: explicit chern polynomial of E}
    \end{align}
    %\todo %%これfirst propertiesのとこにもってく？
\end{lem}
\begin{proof}[\proofname\ of \cref{lemma: Chern polynomial of Ulrich objects}]
    As $\rho(X) =1$, we can write $\ch_1E = e_1H$ and $\ch_2E = e_2H^2= e_2d[\pt]$ with $e_1, e_2\in\ZZ$.
    Then, 
\begin{equation*}
    \chi(E(-k)) = r +(e_1-rkH) + \bigg(e_2d-e_1kH + \frac{k^2}{2}rd\bigg)
\end{equation*}
    and the Hirzebruch-Riemann-Roch Theorem induce 
    \begin{align*}
        r\chi(\OO_X)+d\bigg(-\bigg(\frac{i_X}{2}+1\bigg)e_1 + e_2 + \frac{r}{2}\big(i_X+1\big)\bigg) &= 0,\\
        r\chi(\OO_X)+d\bigg(-\bigg(\frac{i_X}{2}+2\bigg)e_1 + e_2 + r\big(i_X+2\big)\bigg) &= 0.
    \end{align*}
    Thus, we have the result. 
\end{proof}

\begin{comment}
Next, we consider the K3 surfaces.
\begin{thm}[{\cite[Theorem 5.4.3]{book_costa_miroroig_2021_ulrich_bundlesfrom_commutative_algebra_to_algebraic_geometry}}]
    Any smooth K3 surface $(S, \OO_{S}(H))$ admits a rank $2$ special simple Ulrich bundle.
\end{thm}
Let $(S, \OO(1))$ be a smooth K3 surface such that $c_1(\OO(1)) = H$ and $\Pic S = H\ZZ$ and 
$\FC$ a rank $2$ Ulrich bundle on $S$.
Then, 
    $(\rk\FC, c_1(\FC), c_2(\FC)) = \bigg(2, 3H, \frac{5}{2}H^2+4\bigg)$
from \cref{proposition: Chern classed of Ulrich on surface}, and thus 
Mukai vector of $\FC$ is
$v(\FC) = (2, 3H, 2(H^2-1))$.
As $\FC$ is stable, 
\begin{equation}
    \ext^{1}(\FC, \FC) = v(\FC)^2+2 = H^2 + 10. \notag
\end{equation}
\end{comment}

%%%%%%%%%%%%%%%%%%%%%%%%%%%%%%%%%%%%%%%%%%%%%%%%%%%%%%%
%%%%%%%%%%%%%%%%%%%%%%%%%%%%%%%%%%%%%%%%%%%%%%%%%%%%%%%%

\section{Application to the Eisenbud-Schreyer's Question}\label{section: applications}
This \cref{section: applications} states the application to the Eisenbud-Schreyer conjecture to show the evidence of the merit for considering the Ulrich object.
As previously stated, for the sake of convenience, the claim of the Eisenbud-Schreyer's conjecture \cref{Question: the Eisenbud-Schreyer's conjecture} is reiterated here:
any smooth projective variety carries at least an Ulrich bundle.

\subsection{Generators of the Derived Category}\label{subsection: generators of the derived category}

For the notions of the generator of the derived category, 
see \cite{bondal_van_2003_generators_and_representability_of_functors_in_commutative_and_noncommutative_geometry} or \cite[\href{https://stacks.math.columbia.edu/tag/09SI}{Tag 09SI}]{stacks-project} for more details.

Before proceeding to the definition of generators, a notation is introduced: 
\begin{align*}
\langle\Omega\rangle^{\perp} &\coloneqq \{E\in \DC \mid \RHom(X, E) =0, {}^{\forall}X\in\langle\Omega\rangle \}\\
 &= \{E\in \DC \mid \RHom(X, E) =0, {}^{\forall}X\in\Omega \}.
\end{align*}

\begin{dfn}\label{definition: generators of derived category}
    Let $\DC$ be a triangulated category and $G$ be an object of $\DC$.
    Then, $G$ is 
    \begin{enumerate}
        \item \emph{classical generator} if $\langle G\rangle =\DC$, and
        \item \emph{weak generator} (or simply, \emph{generator}) if $\langle G\rangle^{\perp} =0$.
        \end{enumerate}
\end{dfn}
As a remark, we mention the notion of \emph{strong generator}.
$G$ is said to be a strong generator if the entire derived category can be generated in a finite number of steps by operations involving direct summands and extensions, starting with $G$.
(See \cite{bondal_van_2003_generators_and_representability_of_functors_in_commutative_and_noncommutative_geometry} or \cite[\href{https://stacks.math.columbia.edu/tag/09SI}{Tag 09SI}]{stacks-project} for the precise definition.)
By definition, a strong generator is a classical generator.
For the other direction of this implication, the following is known:

\begin{lem}[{\cite[\href{https://stacks.math.columbia.edu/tag/0FXA}{Tag 0FXA}]{stacks-project}}]\label{lemma: classical generator is strong generator}
    Let $\DC$ be a triangulated category that has a strong generator. 
    Let $E$ be an object of $\DC$. If $E$ is a classical generator of $\DC$, then $E$ is a strong generator.
\end{lem}
Thus, we need not distinguish between classical and strong in our cases.

\begin{thm}[see {\cite[Theorem 3.1.4.]{bondal_van_2003_generators_and_representability_of_functors_in_commutative_and_noncommutative_geometry}}]\label{theorem: existence of strong generator}
    Let $X$ be a smooth projective variety. Then, $D^b(X)$ admits a strong generator. 
\end{thm}

The following Theorem says that polarizing a very ample line bundle on $X$ fixes a classical generator of $D^b(X)$ in a sense. 
\begin{thm}[{\cite[Theorem 4.]{orlov_2009_remarks_on_generators_and_dimensions_of_triangulated_categories}}]\label{theorem: explicit form of splitting generator G}
    Let $X$ be a smooth projective variety and $\OO_X(1)$ be a very ample line bundle on $X$. 
    Then, 
    $G = \OO_X\oplus\OO_X(1)\oplus\dots\oplus\OO_X(n)$ 
    is a classical generator on $X$.
\end{thm}

\subsection{Existence for Elliptic Curves}\label{subsection: existence for curves}
In this \cref{subsection: existence for curves}, we will show \cref{Theorem in introduction: curve Ulrich object}.

\begin{prop}
\label{proposition: weak is classical for elliptic curve}
    Let $E$ be an elliptic curve and $G\in D^b(E)$.
    If $G$ weakly generates $D^b(E)$, $G$ classically generates $D^b(E)$.
\end{prop}
\begin{proof}[\proofname\ of \cref{proposition: weak is classical for elliptic curve}]
    We may assume $G$ is the direct sum of semistable sheaves $G_i$ $(1\le i\le m)$ from Atiyah's classification. If $G$ is torsion or semistable, it is not a generator. 
    Thus, we may assume $m\ge2$ and $\mu(G_1)\neq\mu(G_2)$.
    Then $G$ must be a classical generator (see the proof of \cite[Lemma 6.7.]{ballard_favero_katzarkov_2012_orlov_spectra_bounds_and_gaps} for instance).
\end{proof}

\begin{thm}\label{theorem: curve case existence of Ulrich object}
    Let $\Phi_{|\OO_E(1)|}: E\hookrightarrow\PP{N}$ be an elliptic curve with an embedding.
    Then, an Ulrich object exists on $(E, \OO_E(1))$.
\end{thm}
\begin{proof}[\proofname\ of \cref{theorem: curve case existence of Ulrich object}]
    Note that $\OO(1)\in D^b(E)$ cannot classically generate the whole derived cateogy $D^b(E)$.
    Indeed, as $\rk K_0(\langle \OO(1) \rangle) = 1$ by the stability of $\OO(1)$, 
    comparing the rank of the K-group of $\langle \OO(1) \rangle$ and $D^b(E)$, 
    we have 
    \[
    2 = \rk K_0(D^b(E)) > \rk K_{0}(\langle\OO_{E}(1)\rangle) = 1.
    \]
    Thus, $\OO(1)$ cannot classically generate the whole derived category.
    Finally, \cref{proposition: weak is classical for elliptic curve} shows the claim.
\end{proof}
%\begin{comment}

\section{Remarks on Bridgeland Stability of Ulrich Objects}\label{section: Bridgeland stability of Ulrich objects}
The last section is devoted to some remarks on the relationship between the Ulrich objects and stabilities in the derived category, namely the Bridgeland stabilities.

As stated in \cref{subsection: Preliminaries on Ulrich Bundles}, the Ulrich bundles behave well for the stabilities of sheaves. Thus, it predicts that the Ulrich objects behave well for Bridgeland stabilities. 
However, in general, investigating the Bridgeland stabilities is not easy. 

\subsection{Preliminaries on Bridgeland Stability Conditions}\label{subsection: preliminaries on Bridgeland stability condition}

\begin{dfn}\label{definition: stability conditions on triangulated categories}
    Let $\DC$ be a triangulated category, $\AC$ an abelian category 
    and $K_{\mathrm{num}}(\AC)$ and $K_{\mathrm{num}}(\AC)$ be these numerical Grothendieck group respectively.
    \begin{itemize}
        \item A \emph{pre-stability function} on $\AC$ is a group homomorphism 
        $Z: K_{\mathrm{num}}(\AC) \to \CB$ such that for all $A\in\AC\setminus0$, $Z(A)\in \HB^{+} (= \HB \cup \RB_{<0})$. 
        \item A \emph{stability function} $Z$ on $\AC$ is a pre-stability function on $\AC$ which satisfies the Harder-Narasimhan property.
        \item A \emph{Bridgeland stability condition} $\sigma$ on $\DC$ is a pair $(\AC, Z)$, where $\AC\subset \DC$ is an abelian category and $Z$ is a stability function on $\AC$.
    \end{itemize}
    Denote the set of Bridgeland stability conditions on $\DC$ by $\Stab(\DC)$.
    Also, for a smooth projective variety $X$, write $\Stab(X)\coloneqq\Stab(D^b(X))$ for simplicity.
    Note that in \cite{bridgeland_2007_stability_conditions_on_triangulated_categories} Bridgeland have shown that $\Stab(X)$ become a complex manifold with a suitable topology.
\end{dfn}

\begin{prop}\label{proposition: Divisorial stability condition}
    Let $X$ be a surface, $D, H$ be an $\ RB$-divisor such that $H$ is ample.
    Define an abelian category by
        \[
    \mathcal{A}_{D,H} = \{F\in D^{b}(X)\ \mid \ H^{i}(F)= 0\ \text{for}\ i\neq -1,0,\ 
    H^{-1}(F)\in\mathcal{F}_{D,H}, H^{0}(F)\in \mathcal{T}_{D,H}
    \}, 
    \]
    where the pair $(\mathcal{F}_{D,H},\mathcal{T}_{D,H})$ is the torsion pair defined by 
    $$\mathcal{F}_{D,H} = \{ E\in \Coh(X)\ \mid\ E:\text{torsion free and}\ \forall0\neq F\subset E: \mu_{H}(F)\le D.H\},\text{and}$$
    $$\mathcal{T}_{D,H} = \{ E\in \Coh(X)\ \mid\ \forall E\twoheadrightarrow F\neq 0 
    \text{ torsion free}:\ \mu_{H}(F) > D.H\}, $$
    where $\mu_{H}(E) = \frac{c_1(E).H}{\rk(E)H^{2}}$ is the slope function. 
    For $F\in\AC_{D, H}$, set 
    \[
    \displaystyle Z_{D,H}(F) = -\int \exp(-(D+iH))\ch(F).
    \]
    
    Then, there exists the injection:
    \begin{align*}
        N^{1}(X)\times\Amp(X)&\longrightarrow \Stab(X)\\
        (D, H)\hspace{10mm}&\longmapsto\sigma_{D,H} = (Z_{D,H}, \mathcal{A}_{D,H}).
    \end{align*}
\end{prop}

\begin{dfn}\label{definition: geometric stability condition}
   Let $\DC$ be a triangulated category.
   A stability condition $\sigma =(Z, \AC)$ is called \emph{geometric} if 
   skyscraper sheaf $k(x)$ is $\sigma$-semistable of same phase for any $x\in X$. 
\end{dfn}

\begin{prop}\label{proposition: geometric and divisorial}
    Let $X$ be a smooth projective surface over $\CB$.
    Then, the divisorial stability condition $\sigma_{D, H}$ is geometric such that the phase of the skyscraper is $1$.
    Conversely, if $\sigma\in\Stab(X)$ satisfies that $k(x)$ is $\sigma$-stable of phase $1$ for any $x\in X$, then there exists a pair $(D, H)\in N^{1}(X)_{\RB}\times\Amp(X)_{\RB}$ such that 
    $\sigma = \sigma_{D, H}$.
\end{prop}

\subsection{First Examples}\label{subsection: Bridgeland stability first examples}

\begin{eg}\label{example: the case of curve Bridgeland stability of Ulrich object}
    Let $C$ be a smooth projective curve of genus $g>0$ and $\sigma$ be a Bridgeland stability condition on $D^b(C)$.
    An Ulrich object $E$ is trivially semistable from \cref{Theorem: Ulrih bundles are Gieseker stable}.
\end{eg}

\subsubsection{The Projective Plane}\label{subsubsection: Bridgeland stability of projective plane}
From \cref{proposition: Divisorial stability condition}, the divisorial stability condition in $\Stabdiv(\PP{2})$ is determined by the pair in the upper half-plane $(s,t)\in\HB$.
Denote $\sigma_{sH,tH}$ by $\sigma_{s,t}$.

\begin{thm}[{\cite[Theorem 5.4.]{arcara_miles_2016_bridgeland_stability_of_line_bundles_on_surfaces}}]\label{Theorem: Bridgeland stability of structure sheaf for surface}
    Let $S$ be a surface that does not have any negative self-intersection curve, 
    and $\sigma=(Z, \AC)\in\Stabdiv(X)$ such that $\OO_S\in\AC$.
    Then, $\OO_S$ is stable with respect to $\sigma$.
\end{thm}
In particular, the structure sheaf $\OO_{\PP{2}}$ of the projective plane $\PP{2}$ is 
Bridgeland stable for any $\sigma_{sH, tH}$ $(s<0)$. 

The dual version of the \cref{Theorem: Bridgeland stability of structure sheaf for surface} is the following:
\begin{prop}[{\cite[Proposition 6.3.]{arcara_miles_2016_bridgeland_stability_of_line_bundles_on_surfaces}}]\label{proposition: Bridgeland stability of structure sheaf for shifted surface}
    Let $S$ be a surface that does not have any negative self-intersection curve, 
    and $\sigma=(Z, \AC)\in\Stabdiv(X)$ such that $\OO_S[1]\in\AC$.
    Then, $\OO_S[1]$ is stable with respect to $\sigma$.
\end{prop}

Thus, the shifted structure sheaf $\OO_{\PP{2}}[1]$ is 
Bridgeland stable for any $\sigma_{sH, tH}$ $(s>0)$. 

As a corollary of the results in \cite{arcara_miles_2016_bridgeland_stability_of_line_bundles_on_surfaces}, we have the following statement:
\begin{prop}\label{proposition: Ulrich objects on P2 are Bridgeland semistable}
    Let $E$ be an Ulrich object on $\PP{2}$.
    Assume that there exists an integer $k\in\ZZ$ and a divisorial stability condition $\sigma_{s, t}=(Z_{s, t}, \AC_{s, t})$ on $D^b(\PP{2})$ such that $E[k]\in \AC_{s, t}$.
    Then, $E$ is $\sigma_{s, t}$-semistable.
\end{prop}
\begin{proof}[\proofname\ of \cref{proposition: Ulrich objects on P2 are Bridgeland semistable}]
    Arcara-Miles says $\OO_{\PP{2}}$ is $\sigma_{s, t}$-stable. 
    An Ulrich object on $\PP{2}$ is the form of 
    \[
    E\cong\bigoplus_{i\in\ZZ}\OO_{\PP{2}}^{n_i}[-i].
    \]
    by \cref{Proposition: Ulrich object characterization}.
    By the definition of $\AC_{s, t}$, 
    only one $n_i\neq 0$ if $E[k]\in\AC_{s, t}$. 
\end{proof}

\subsubsection{Quadrics}\label{subsubsection: Bridgeland stability quadrics}
The Spinor bundles on the quadric surface $Q^{2}\cong\PP{1}\times\PP{1}$ are 
$\OO(1, 0)$ and $\OO(0, 1)$ via the identification. 
Thus, Spinor bundles on $Q^2$ is $\sigma$-stable for $\sigma\in\Stab_{\mathrm{div}}(Q^2)$ up to shift from {\cite[Theorem 5.4. and Proposition 6.3.]{arcara_miles_2016_bridgeland_stability_of_line_bundles_on_surfaces}}.

Next, Let $Q^3$ be the quadric threefold.
The Spinor bundle $S$ is an object in the Kuznetsov component $\Ku(Q^3)$.
{\cite[Lemma A.1.]{preprint_song_2024_moduli_of_stable_sheaves_on_quadric_threefold}} claims that the Spinor bundle is $\sigma$-stable for any 
\[
\sigma\in\KC \coloneqq \Big\{\sigma\left(\alpha, -\frac{1}{2}\right)\cdot \widetilde{\GL}_{2}^{+}(\RB)\Big\}\subset \Stab(\Ku(Q^3)).
\]
See {\cite[Section 4.]{preprint_song_2024_moduli_of_stable_sheaves_on_quadric_threefold}} for more precise definitions.

\subsubsection{Bridgeland Stability of Ulrich Objects on Surfaces}\label{subsection: Bridgeland stability of Ulrich objects on surfaces}

%\begin{comment}
\begin{lem}\label{lemma: Bridgeland stability function of Ulrich object}
    With the above settings, 
    \begin{align}
    Z_{s, t}(E) &= \bigg\{-r\chi(\OO_X) +\frac{rd}{4}(i_{X}^{2}+3i_X+4) + 
    \frac{rd}{2}\left(s^2-t^2+(i_X+3)s\right)\bigg\} \notag \\
    &\quad + \bigg\{\frac{rd}{2}(2s+i_X+3)t\bigg\}\sqrt{-1}.
    \end{align}
\end{lem}
\begin{proof}[\proofname\ of \cref{lemma: Bridgeland stability function of Ulrich object}]
    We have the calculation using \cref{lemma: Chern polynomial of Ulrich objects}.
\end{proof}
%\end{comment}

\begin{prop}\label{proposition: Yoneda extension is not in the heart}
Let $X$ be a surface and
$E\in D^b(X)$ be an Ulrich object with respect to an embedding $\OO_{X}(1)$.

Assume that $E$ is contained in a tilting heart $\AC_{D, H}$ for some $(D, H)\in N^{1}(X)_{\RB}\times \Amp(X)_{\RB}$.
Then, $E$ is a shift of an Ulrich sheaf.
\end{prop}
\begin{proof}[\proofname\ of \cref{proposition: Yoneda extension is not in the heart}]
Assume $\HC^i(E)\neq 0$ for $i=-1, 0$.
Note that $\HC^i(E)$ are both Ulrich bundles by \cref{Proposition: Ulrich object characterization} and 
semistable by \cref{Theorem: Ulrih bundles are Gieseker stable}.
As the slopes of $\HC^i(E)$ coincide, it contradict to the construction of $\AC_{D, H} = \langle\FC_{D, H}[1], \TC_{D, H}\rangle$.
\end{proof}

\begin{question}\label{question: Bridgeland stability of Ulrich objects}
    Let $X$ be a smooth surface over $\CB$. 
    Is there any geometric stability condition $\sigma = (Z, \AC)\in\Stab(X)$ 
    and an Ulrich object $E$ such that $E\in\AC$
    that is not $\sigma$-semistable?
\end{question}

\bibliographystyle{amsalpha}
\bibliography{bibtex_tyoshida}
\end{document}